\documentclass[11pt]{amsart}
\usepackage[margin=2.7cm, marginpar=1cm]{geometry}

\usepackage[T1]{fontenc}
\usepackage{comment}
\usepackage{appendix}
\usepackage{graphicx}
\usepackage{pdfpages}
\usepackage{nicematrix}
\newtheorem{theorem}{Theorem}[section]
\theoremstyle{definition}

\theoremstyle{remark}
\newtheorem{remark}[theorem]{Remark}
\usepackage{subcaption}
\usepackage{hyperref}
\usepackage[font=small]{caption}

\usepackage{epic}
\usepackage{import}
\usepackage{pdfpages}

\theoremstyle{plain}
\newtheorem{thm}{Theorem}

\newtheorem{cor}[thm]{Corollary}

\theoremstyle{definition}

\newtheorem{conj}[thm]{Conjecture}

\title{On the Legendrian invariant in knot lattice homology}
\author{Sarah Zampa}
\date{}

\newcommand{\F}{\mathbb{F}}
\newcommand{\Z}{\mathbb{Z}}
\newcommand{\PP}{\mathbb{P}}
\newcommand{\Q}{\mathbb{Q}}
\newcommand{\ie}{\textit{i}.\textit{e}., }
\newcommand{\CFlattice}{\mathbb{CF}}
\newcommand{\CFKlattice}{\mathbb{CFK}}
\newcommand{\HFlattice}{\mathbb{HF}}
\newcommand{\HFKlattice}{\mathbb{HFK}}
\newcommand{\gammavo}{\Gamma_{v_0}}
\newcommand{\Gvo}{G_{v_0}}
\newcommand{\leg}{\mathcal{L}(L)}

\DeclareMathOperator{\Vertgraph}{Vert}
\DeclareMathOperator{\gr}{gr}
\DeclareMathOperator{\Char}{Char}

\DeclareMathOperator{\HF}{HF}

\DeclareMathOperator{\HFhat}{\widehat{HF}}

\begin{document}

\begin{abstract}
The Ozsv\'{a}th-Szab\'{o} contact invariant $c^+(\xi)\in\mathrm{HF}^+(-Y)$ of the link of a normal surface singularity equipped with its canonical contact structure $(Y,\xi)$ was transposed to lattice homology theory by Bodn\'{a}r-Plamenevskaya. When considering a transverse algebraic knot $L$ in the link, the chain complex computing $\mathrm{HF}^+(-Y)$ can be equipped with an Alexander grading, and we can define an element $\mathcal{L}(L)$ in the bigraded theory $\mathrm{HFK}^+(-Y,L)$, which maps to the contact element by forgetting the filtration. We show that the Alexander grading (as defined by Ozsv\'{a}th-Stipsicz-Szab\'{o}) of this element is invariant under all blow-ups of the underlying plumbing graph. Furthermore, we utilize the fact that for specific types of blow-ups, the resulting lattice chain complexes are filtered chain homotopic and the chains maps map this element in one chain complex to the other, thereby providing a partial combinatorial description of the Legendrian invariant.
\end{abstract}

\maketitle

\section{Introduction}

Introduced by Ozsv\'{a}th and Szab\'{o} in \cite{OS04}, Heegaard Floer homology provides powerful invariants for $3$-manifolds and knots, particularly in the study of contact and Legendrian topology. A combinatorial counterpart to this theory known as \emph{lattice homology} was developed by N\'{e}methi (see \cite{Ne05,Ne08}) for links of isolated surface singularities. In this geometric setting, a resolution of the singularity is encoded by a weighted graph $G$ called a \emph{plumbing graph}, whose vertices correspond to the irreducible exceptional components of a good resolution, and where an edge is placed between two transverse intersecting curves. To each vertex, an integer weight corresponding to the self-intersection of the considered component is associated. The $4$-manifold $X_G$ is obtained by plumbing disk bundles over surfaces of genus $g$ as instructed by the graph, and the $3$-manifold $Y$ is the boundary of $X_G$. We will restrict ourselves to the case of rational homology spheres, which is equivalent to the plumbing graph being a disjoint union of trees of spheres.
Note that in the case of a link of a normal surface singularity, the plumbing graph is connected. In the following, we will consider more general cases, when the graph is a union of connected components, which at the level of $3$-manifolds corresponds to the connected sum.

The plumbing graph is not determined by the singularity, it is an invariant of the resolution of the singularity only up to some moves called \emph{blow-ups} and \emph{blow-downs}. These moves correspond to taking the connected sum of $X_G$ with $\overline{\mathbb{CP}}^2$ for the blow-up, and the blow-down corresponds to the inverse operation. Note that these operations do not affect the boundary of the $4$-manifold and so leave $Y$ unchanged.

Bodn\'{a}r and Plamenevskaya in \cite{BP21} successfully transposed the contact invariant $c^+(\xi)$ into the lattice setting. In this paper, we extend their work to the Legendrian/transverse context. We consider a plumbing graph $\gammavo$ with a distinguished vertex $v_0$ representing a transverse knot $L \subset (Y,\xi)$. Our goal is to define the Legendrian invariant $\leg \in \HFKlattice^-(\gammavo)$ combinatorially.

\begin{thm}
\label{thm:mainthm}
    The Alexander grading of the element $\leg \in \HFKlattice^-(\gammavo)$ is invariant under all five types of blow-ups depicted in Figure \ref{fig:blowups-general}.
\end{thm}

While it is established that two plumbing graphs representing the same $3$-manifold $Y$ are related by a sequence of blow-ups/downs, it is yet to be proven that the same applies to the pair $(Y,L)$ where $L$ is a transverse knot in $(Y,\xi)$ represented by an unweighted vertex.

\begin{conj}
     Let $\gammavo$ and $\Gamma'_{v'_0}$ be two plumbing graphs representing the same triplet $(Y,\xi,L)$. Then, the two graphs are related by the blow-ups depicted in Figure \ref{fig:blowups-general}.
\end{conj}

Utilizing some filtered chain homotopies from \cite{OSS14a} between the filtered chain complexes of two different graphs representing the same pair $(Y,K)$, we obtain the immediate result:

\begin{cor}
   If the two graphs are related by blow-ups of type I, IV and V exclusively, then the element $\leg$ is an invariant of the triplet $(Y,\xi,L)$.
\end{cor}

We conjecture that the element $\leg$ is invariant under all types of blow-ups.

The paper is organized as follows. In Section \ref{sec:background}, we review the technical construction of N\'{e}methi's lattice homology, the specific combinatorial moves of the blow-up, the definition of the knot lattice complex with its Alexander grading, and the contact invariant; we also define $\leg$. Section \ref{sec:alexandergrading} is dedicated to the proof of invariance of the Alexander grading under blow-ups, as well as discuss the filtered chain homotopies that exist for three of the blow-up types.

\vspace{5mm}
\noindent\textbf{Acknowledgements:} The author wishes to thank Andr\'{a}s Stipsicz for his help and guidance. The author was partially supported by the ERC Advanced Grant KnotSurf4d and by the NKFIH Grant K146401.

\section{Background material}
\label{sec:background}

We review the basics of N\'{e}methi's lattice homology \cite{Ne05,Ne08} and the knot lattice viewpoint \cite{OSS14a}. In the following, we only consider negative definite plumbing trees of spheres, or equivalently singularities with rational homology sphere links.

\subsection{Plumbing graphs}

Let $G$ be a plumbing graph associated with a good resolution of a normal surface singularity. We will always have that $G$ is a union of \emph{negative definite trees}. The vertices $v\in\Vertgraph(G)$ are labeled by integer weights denoted by $v^2$ called \emph{framings}.

To study knots, we consider a decorated graph $\gammavo$ where all vertices are framed except for a distinguished vertex $v_0$. This vertex corresponds to a transverse knot $L$ in the contact $3$-manifold $(Y,\xi)$ where $\xi$ is the canonical contact structure. We denote by $G_{v_0}$ the framed graph obtained from $\gammavo$ by adding a specific framing to $v_0$.

\subsection{Blow-up operations}

We focus on the blow-up operation, which corresponds to the connected sum with $\overline{\mathbb{CP}}^2$. There are several combinatorial types of blow-ups depending on the location of the operation.

\begin{figure}%[ht]
     \centering
     \begin{subfigure}[b]{0.5\textwidth}
         \centering
         \includegraphics[width=\textwidth]{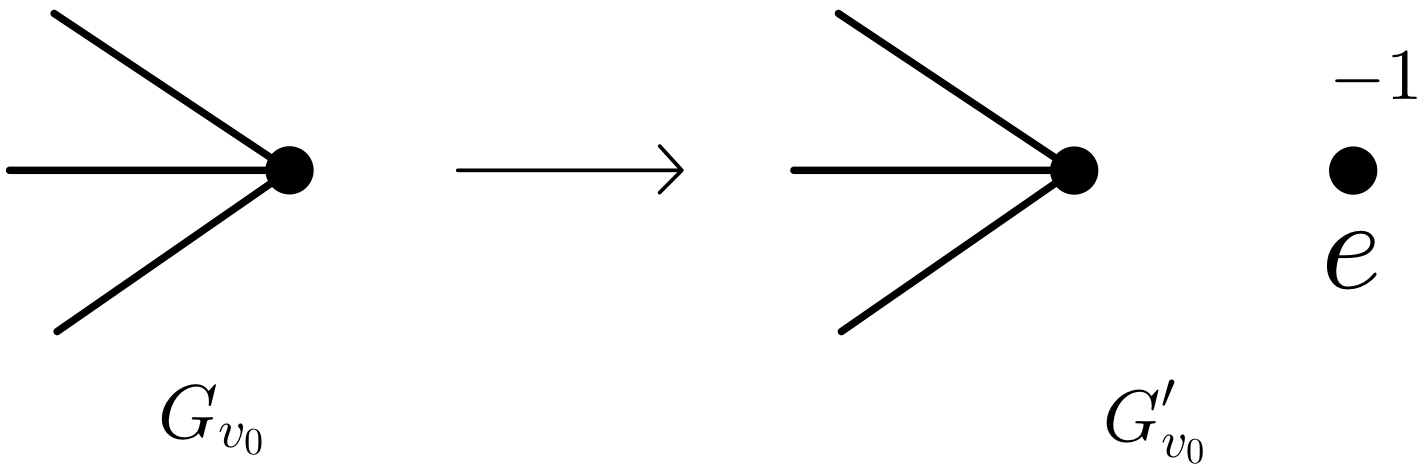}
         \caption{Connected sum with $S^3$.}
        \label{fig:blowupI}
     \end{subfigure}
     \hfill
     \begin{subfigure}[b]{0.5\textwidth}
         \centering
         \includegraphics[width=\textwidth]{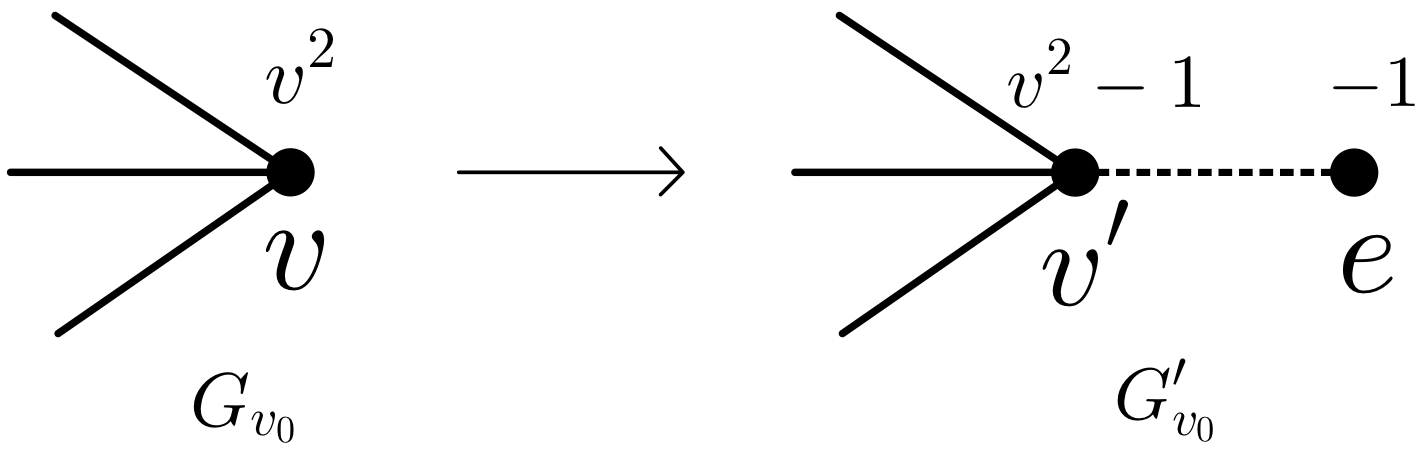}
         \caption{Blow-up of a weighted vertex.}
        \label{fig:blowupII}
     \end{subfigure}
     \hfill
     \begin{subfigure}[b]{0.79\textwidth}
         \centering
         \includegraphics[width=\textwidth]{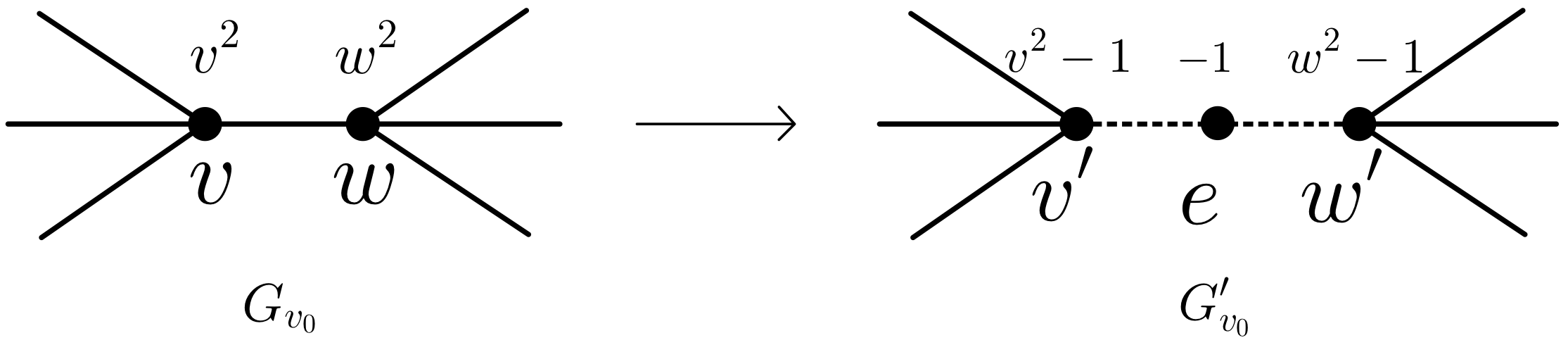}
         \caption{Blow-up of an edge between two weighted vertices.}
        \label{fig:blowupIII}
     \end{subfigure}
     \hfill
     \begin{subfigure}[b]{0.65\textwidth}
         \centering
         \includegraphics[width=\textwidth]{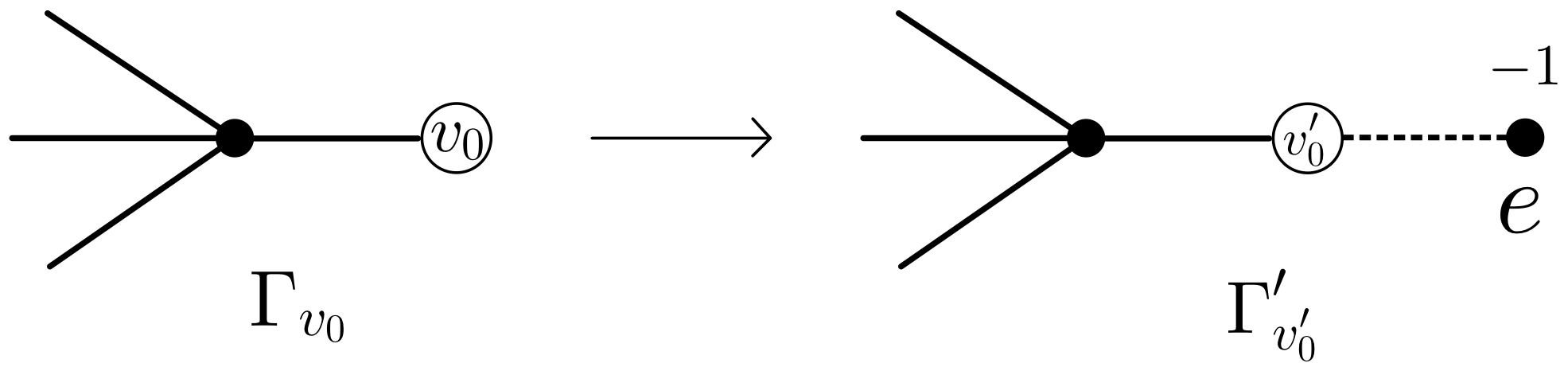}
         \caption{Blow-up of the unweighted vertex.}
        \label{fig:blowupIV}
     \end{subfigure}
     \hfill
     \begin{subfigure}[b]{0.65\textwidth}
         \centering
         \includegraphics[width=\textwidth]{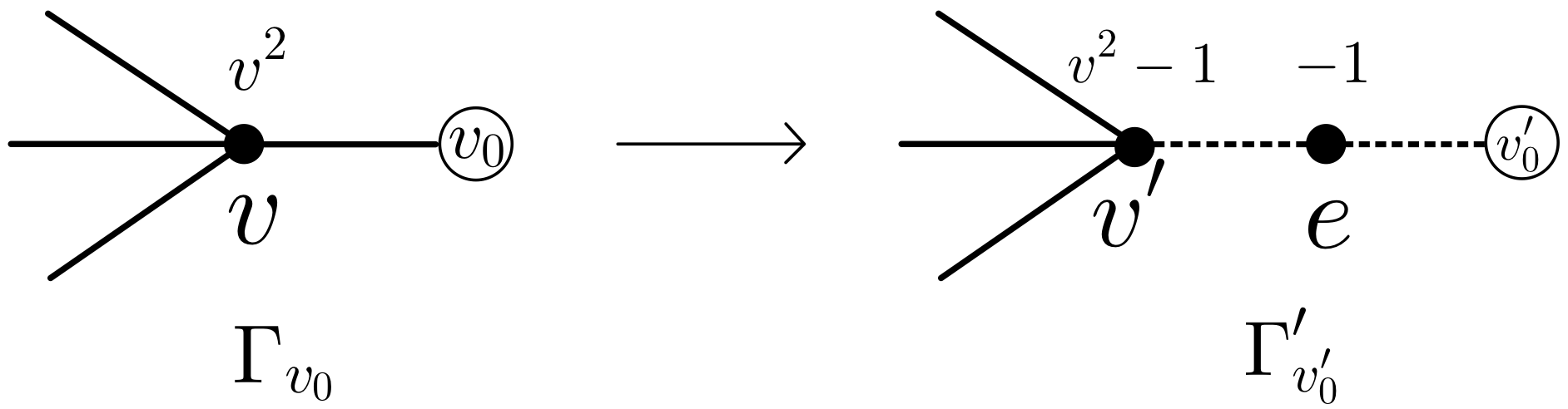}
         \caption{Blow-up of an edge between a weighted vertex and the unweighted vertex.}
        \label{fig:blowupV}
     \end{subfigure}
        \caption{The five different types of blow-ups.}
        \label{fig:blowups-general}
\end{figure}

First, we can consider blowing-up away from the curves as depicted in Figure \ref{fig:blowupI}, in which case the resulting graph is the disjoint union of the original graph together with a new vertex $e$ with weight $-1$; this type of blow-up is usually called a \emph{generic} blow-up. We can then consider blowing-up at a single curve, which we distinguish in two cases, as shown in Figures \ref{fig:blowupII} and \ref{fig:blowupIV}. Either we blow-up a curve corresponding to a weighted vertex $v$, in which case we add an edge between $v$ and the new vertex $e$, and lower the framing $v^2$ by one. Or we can blow-up the unique curve without a framing, in which case we only add the edge. Finally, we can blow-up at the intersection between two curves, which corresponds to blowing-up an edge between two vertices. We can either do so between two weighted vertices $v$ and $w$ as in Figure \ref{fig:blowupIII}, in which case both framings are lowered by $1$, or we can do so on an edge incident to the unweighted vertex as in Figure \ref{fig:blowupV}, in which case we lower the framing of the weighted vertex $v$.

\subsection{Lattice homology}

Let $X_G$ be the $4$-manifold associated to $G$. The lattice chain complex is defined as
\[
\CFlattice^\infty(G)\cong\F[U,U^{-1}]\langle \Char(G)\times \PP(\Vertgraph(G)) \rangle,
\]
where $\Char(G)=\{K:H_2(X_G;\Z)\rightarrow\mathbb{Z}\, \vert \,K(x)\equiv x\cdot x \pmod{2}\}$ is the set of characteristic cohomology elements of $H^2(X_G;\Z)$. This means that generators are pairs $[K,E]$ with $K\in \Char(G)$ and $E \subset \Vertgraph(G)$.
By abuse of notation, we will denote by $v\in H_2(X_G;\mathbb{Z})$ the homology class corresponding to the vertex $v\in\Vertgraph(G)$, and by $v^*\in H^2(X_G,Y;\mathbb{Z})$ its Poincaré dual.
The lattice chain complex $\CFlattice^\infty (G)$ splits according to the Spin$^c$ structures on $Y_G =\partial X_G$:
\[
\CFlattice^\infty (G)= \sum_{\mathbf{s}\in Spin^c(Y_G)} \CFlattice^\infty (G,\mathbf{s})
\]
Let us define a map $\partial$ on the chain complex. First, we define the function $g:\Char(G)\rightarrow\mathbb{N}$, called the \textit{minimal $G$-weight}, as $g([K,E])=\min\{f([K,I]) \,\vert\, I\subset E\}$, where $$2f([K,I])= \left(\sum_{v\in I} K(v)\right)+\left(\sum_{v\in I}v\right)\cdot\left(\sum_{v\in I}v\right).$$
Now, we can define the map
\[
\partial [K,E]=\sum_{v\in E} U^{a_v [K,E]} \otimes [K,E-v] + \sum_{v\in E} U^{b_v [K,E]} \otimes [K+2v^*,E-v]
\]
where $a_v [K,E]=g([K,E-v])-g([K,E])$ and $b_v [K,E]=\frac{1}{2}(K(v)+v^2)+g([K+2v^*,E-v])-g([K,E])$. It is a boundary map (\ie $\partial\circ\partial=0$), see \cite[Lemma 2.2]{OSS14a}.

The chain complex $(\CFlattice^-(G),\partial)$ is then defined to be the subcomplex of $(\CFlattice^\infty(G),\partial)$ for which the generators $U^j\otimes [K,E]$ are such that $j\geq 0$.
In particular, the infinity complex is a module over $\mathbb{F}[U,U^{-1}]$, while the minus version is a module over $\mathbb{F}[U]$.

There is a natural grading given by the exponent of $U$ in the expression of an element, and another grading $\gr : \CFlattice^\infty(G) \longrightarrow \Q$ called the \textit{Maslov grading}, defined as
\[
\gr(U^j\otimes [K,E])=-2j+2g([K,E])+\vert E\vert +\frac{1}{4}(K^2+\vert \Vertgraph(G)\vert),
\]
where $K^2$ is defined as the square of $nK\in H^2(X_G,Y;\mathbb{Z})$ divided by $n^2$, and $n=\vert H_1(X;\mathbb{Z})\vert$.
The lattice homology $\HFlattice^\infty(G)$ is defined to be the homology of the chain complex $(\CFlattice^\infty(G),\partial)$ (and $\HFlattice^-(G)$ is defined in the same manner).

\subsection{Knot lattice homology}

In order to define the knot lattice chain complexes, we need to introduce a filtration. For a generator $[K,E]$, let
$A$ be the \textit{Alexander grading} defined by
\[
A([K,E])=\frac{1}{2}(L(\Sigma)+\Sigma^2),
\]
where $L=L_{[K,E]}\in \Char(\Gvo)$ is the unique element such that
\begin{itemize}
    \item $L\vert_G=K$ and
    \item $a_{v_0}[L,E\cup\{v_0\}]=b_{v_0}([L,E\cup\{v_0\}])=0$,
\end{itemize}
and $\Sigma=v_0+\sum_{j=1}^n a_j\cdot v_j$ is the homology class in $H_2(X_{\Gvo};\Q)$ defined by $v_j\cdot\Sigma=0$ for all $j\neq 0$. This grading is extended to elements of the form $U^j\otimes[K,E]$ by
\[
A(U^j\otimes[K,E])=-j+A([K,E])
\]
where $j\in\mathbb{Z}$.

This grading naturally induces a filtration $\{\mathcal{F}_i\}$ on $\CFlattice^\infty(G)$; by intersecting it with $\CFlattice^-(G)$, we also get a filtration $A$ on this subcomplex. Together with this filtration, the subcomplex $(\CFlattice^-(G), \partial, A)$ is a filtered chain complex.
The homology of its associated graded object defines the knot lattice homology $\HFKlattice^-(\gammavo)$ (\ie $\CFKlattice^-(G)$ is $\CFlattice^-(G)$ together with the filtration induced by $A$).
The associated graded object for $(\CFlattice^-(G), \partial, A)$ is the graded object that is exactly $\mathcal{F}_i/\mathcal{F}_{i-1}$ in degree $i$.
This means that the boundary map respects the filtration induced by $A$.

\subsection{The connected sum formula}
\label{subsec:connsum}

A tool which will prove to be important later on is to consider the connected sum of two graphs with distinguished vertices. Namely, if $\gammavo$ and $\Gamma'_{w_0}$ are two graphs with unweighted vertices $v_0$ and $w_0$, we can consider their disjoint union, with the identification $v_0=w_0$. The new graph is denoted by $\Delta_{(v_0=w_0)}$, and if $G=\gammavo-v_0$ and $G'=\Gamma'_{w_0}-w_0$, the Alexander grading of an element $[K_1,E_1]\otimes[K_2,E_2]\in\CFlattice^\infty(G\cup G')$ satisfies $$A_\#([K_1,E_1]\otimes[K_2,E_2])=A_{v_0}([K_1,E_1])+A_{w_0}([K_2,E_2]),$$where the different Alexander gradings refer to the gradings associated to the different graphs.

\subsection{The contact invariant}

In Heegaard Floer theory, the \emph{contact invariant} $c(\xi)$ is a distinguished element living in the homology group $\HFhat(-Y)$ (or the plus version $\HF^+(-Y)$) of the closed oriented $3$-manifold $Y$ equipped with a contact structure $\xi$.
These groups are invariants of a closed oriented $3$-manifold defined from a Heegaard diagram of the orientation-reversed manifold $-Y$, and come in several flavors depending on the ring on which they are defined.
First defined by Ozsv\'{a}th and Szab\'{o} in \cite{OS02}, $c(\xi)$ is constructed using an open book decomposition compatible with $\xi$. It has been reformulated by Honda, Kazez and Mati\'{c} in \cite{HKM07} as follows.
Concretely, we start with an open book decomposition and consider the associated Heegaard diagram for $Y$; the generator is the unique intersection point between some specific curves that lie in a position dictated by the monodromy.
We can consider the homology class of this cycle, which depends only on the isotopy class of $\xi$.
Its most fundamental property is that it detects tightness:
\begin{thm}[Ozsv\'{a}th-Szab\'{o}, \cite{OS02}]%Theorem 1.4
    If $c(\xi)\neq 0$, then the contact structure $\xi$ is tight.
\end{thm}
Moreover, the contact invariant is non-vanishing if $\xi$ is induced by a Stein filling, and lies in a specific grading, which makes it for a good tool to study fillability.

In the setting of the Lisca-Ozsv\'{a}th-Stipsicz-Szab\'{o} construction \cite{LOSS}, the same open book decomposition gives a more refined invariant for a null-homologous Legendrian knot $L$. Viewing the knot on a page, we can form the compatible doubly pointed Heegaard diagram and take the generator associated with the arcs in the knot Floer complex. The homology class defines the Legendrian invariant, and after forgetting the filtration we get back $c(\xi)$.

Focusing now on links of surface singularities, the computations of Heegaard Floer homology can simplify with the help of lattice homology, since it is a purely combinatorial homology theory. Notably, Bodn\'{a}r-Plamenevskaya have in \cite{BP21} translated the contact invariant to this combinatorial point of view. They have shown that the invariant can be viewed as a specific cycle in the lattice chain complex, namely the \emph{anticanonical class} $K_0\in\Char(G)$, which acts on elements of $\Vertgraph(G)$ by
\[
K_0(v)=v^2+2
\]
and derives its name from the fact that $K_0=-K_{can}$, where the canonical class $K_{can}$ is defined by N\'{e}methi in \cite{Ne05} to be the first Chern class of the canonical line bundle.

\section{Invariance of the Alexander grading under blow-ups}
\label{sec:alexandergrading}

We are now ready to show that the Alexander grading of the anticanonical class $K_0=[K_0,\emptyset]$ is invariant under blow-ups, for each of the five types of blow-ups described previously.
In the following, we will use the fact that $\Sigma^2=v_0\cdot\Sigma$.
\begin{thm}
    The Alexander grading of the anticanonical class $K_0=[K_0,\emptyset]$ is invariant under blow-ups.
\end{thm}

\begin{proof}
    The proof will go through the five different types of blow-ups.
    
    Let us denote by $L_0$ the extension of $K_0$ to $\Char(\Gvo)$ in the definition of the Alexander grading. Consider $\Sigma=v_0+\sum_{j=1}^{n-1} a_j\cdot v_j$ the homology class defined previously on the original graph $\Gvo$, and $\Sigma'=v'_0+\sum_{j=1}^n a'_j\cdot v'_j$ on the blown-up graph $\Gvo'$, \ie $v_n=e$ is the new vertex with $v^2_n=-1$ (we will write $v_n$ instead of $v'_n$ as there is not confusion for this vertex).

In the case of a generic blow-up (the disjoint union of our original graph $\Gvo$ with a $-1$-framed vertex), we immediately see that all the coefficients $a'_j=a_j$ for $j=1,\dots,n-1$, and $a_n=0$. This implies that $(\Sigma')^2=\Sigma^2$ and that $L_0(\Sigma')=L_0(\Sigma)$, and in turn that $A(K_0)$ is unchanged after generic blow-up.
    
Suppose now that we blow up at the weighted vertex $v_k$. We will see that $\Sigma'=\Sigma+a_k\cdot v_n$ (in other words, the coefficient of the new vertex in the homology class is equal to the coefficient of the blown-up vertex). Indeed, for any $j\neq 0,k,n$, the condition that $v'_j\cdot \Sigma'=0$ will imply that the coefficients $a'_j=a_j$ for any $j\neq k,n$. For the last coefficients, we first have $v_n\cdot\Sigma'=a'_k(v_n v_k)+a'_n v^2_n=a'_k-a'_n=0$ and so $a'_n=a'_k$. To see that the coefficient of $v_k$ is left unchanged, recall that the framing of $v_k$ in $\Gvo'$ is one less than its framing in $\Gvo$.
This means that
    \begin{align*}
        v'_k\cdot\Sigma'=0&\iff \sum_{v'_j \text{ incident to } v'_k}a'_j(v'_k v'_j) +a'_k(v'_k)^2+a_n(v'_k v_n)=0\\
        &\iff \sum_{v_j \text{ incident to } v_k}a_j(v_k v_j) +a'_k(v^2_k-1)+a_n(v_k v_n)=0\\
        &\iff \sum_{v_j \text{ incident to } v_k}a_j +a'_k(v^2_k-1)+a'_k=0\\
        &\iff \sum_{v_j \text{ incident to } v_k}a_j +a'_k v^2_k=0\\
        &\iff v_k\cdot\Sigma=0
    \end{align*}
so the conditions on the coefficient $a_k$ are unchanged after the blow-up, such that $a'_k=a_k$. Now, we first notice that $(\Sigma')^2=v_0\cdot\Sigma'=v_0\cdot\Sigma=\Sigma^2$ since the new vertex $v_n$ is not incident to $v_0$, and that
\begin{align*}
        L_0(\Sigma')&=L_0(v_0)+\sum_{j\neq k,n}a'_j(2+(v'_j)^2)+a'_k(2+(v'_k)^2)+a'_n(2+(v'_n)^2)\\
        &=L_0(v_0)+\sum_{j\neq k,n}a_j(2+v^2_j)+a_k(2+(v^2_k -1))+a_n(2+v^2_n)\\
        &=L_0(v_0)+\sum_{j\neq k,n}a_j(2+v^2_j)+a_k(1+v^2_k)+a_k\\
        &=L_0(\Sigma)
\end{align*}
which shows that the Alexander grading of $K_0$ is left unchanged after this blow-up.

The proof for the blow-up of an edge connecting two weighted vertices goes similarly. If $v_{k_1}$ and $v_{k_2}$ are two such vertices, then we will see that the homology class is given by $\Sigma'=\Sigma+(a_{k_1}+a_{k_2})v_n$. By the same argument as before, all the coefficients $a'_j=a_j$ for $j\neq k_1, k_2, n$, and the condition $v_n\cdot\Sigma'=0$ implies that $a_{k_1}(v_n v_{k_1})+a_{k_2}(v_n v_{k_2})+a_n v^2_n=a_{k_1}+a_{k_2}-a_n=0$, so that the coefficient of the new vertex is equal to $a_{k_1}+a_{k_2}$. Besides, considering $v_{k_1}$ (and similarly for $v_{k_2}$), remembering that the vertices $v_{k_1}$ and $v_{k_2}$ are incident in $\Gvo$ but not in $\Gvo'$, we have

\begin{align*}
    v'_{k_1}\cdot\Sigma'&=\sum_{v'_j \text{ incident to } v'_{k_1}}a'_j(v'_{k_1} v_j) +a'_{k_1}(v'_{k_1})^2\\
    &=\sum_{v'_j \text{ incident to } v'_{k_1}}a'_j +a'_{k_1}(v^2_{k_1}-1)\\
    &=\sum_{\substack{v'_j \text{ incident to } v'_{k_1}\\j\neq n}}a'_j +a_n+a'_{k_1}(v^2_{k_1}-1)\\
    &=\sum_{\substack{v_j \text{ incident to } v_{k_1}\\j\neq n}}a_j-a_{k_2} +a_n+a'_{k_1}(v^2_{k_1}-1)\\
    &=\sum_{\substack{v_j \text{ incident to } v_{k_1}\\j\neq n}}a_j +a_{k_1}+a'_{k_1}(v^2_{k_1}-1).
\end{align*}

Since $v_j\cdot\Sigma=0$ for every $j\neq 0,n$ we have that
\[
\sum_{\substack{v_j \text{ incident to } v_{k_1}\\j\neq n}}a_j +a_{k_1}v^2_{k_1}=0.
\]
This means that the condition $v'_{k_1}\cdot\Sigma'=0$ is equivalent to $(a'_{k_1}-a_{k_1})(v^2_{k_1}-1)=0$, so that the coefficients $a'_{k_1}=a_{k_1}$ (and similarly $a'_{k_2}=a_{k_2}$). Finally, we see that $(\Sigma')^2=v_0\cdot\Sigma'=v_0\cdot\Sigma=\Sigma^2$, and
\begin{align*}
        L_0(\Sigma')&=L_0(v_0)+\sum^{n-1}_{\substack{j=1\\j\neq k_1,k_2}}a_j(2+v^2_j)+a_{k_1}(2+(v^2_{k_1} -1))+a_{k_2}(2+(v^2_{k_2} -1))+a_n(2+v^2_n)\\
        &=L_0(v_0)+\sum^{n-1}_{\substack{j=1\\j\neq k_1,k_2}}a_j(2+v^2_j)+a_{k_1}(1+v^2_{k_1})+a_{k_2}(1+v^2_{k_2})+(a_{k_1}+a_{k_2})\\
        &=L_0(v_0)+\sum^{n-1}_{j=1}a_j(2+v^2_j)\\
        &=L_0(\Sigma)
\end{align*}
so we are done for this case.

Now we will consider the case of a blow-up at the distinguished vertex $v_0$, and for more comprehensible computations, we will denote by $v'_0$ this vertex in $\Gvo'$; this means that for a framing $v^2_0$ on $v_0$, then $v'_0$ has framing $v^2_0 -1$. As before, the coefficients $a'_j=a_j$ will remain unchanged for $j\neq n$, and we find $a'_n$ by
    \begin{align*}
        v_n\cdot\Sigma'=v_nv'_0+a'_nv^2_n=1-a'_n=0
    \end{align*}
and so $a'_n=1$.
    We can now compute on one hand
    \begin{align*}
        (\Sigma')^2=v'_0 \cdot\Sigma'&=\sum_{v'_j \text{ incident to } v'_0} a'_j(v'_0 v'_j)+(v'_0)^2\\
        &=\sum_{v'_j \text{ incident to } v'_0} a_j+(v^2_0-1)\\
        &=\sum_{v_j \text{ incident to } v_0} a_j+a'_n+(v^2_0-1)\\
        &=\sum_{v_j \text{ incident to } v_0} a_j+v^2_0\\
        &=v_0 \cdot\Sigma=\Sigma^2,
    \end{align*}
    while
    \begin{align*}
        L_0(\Sigma')&=\sum_{j=1}^n a'_j(2+(v'_j)^2) + (2+(v'_0)^2)\\
        &=\sum_{j=1}^{n-1} a'_j(2+(v'_j)^2)+a'_n(2+v^2_n) + (2+v^2_0-1)\\
        &=\sum_{j=1}^{n-1} a_j(2+v^2_j)+1 + (1+v^2_0)\\
        &=\sum_{j=1}^{n-1} a_j(2+v^2_j) + (2+v^2_0)\\
        &=L_0(\Sigma)
    \end{align*}
    and so we are done.

    For the final case, namely the blow-up of an edge connecting the distinguished vertex $v_0$ with a weighted vertex $v_k$, we can see that the coefficient $a_n$ of the new vertex is equal to $a_k +1$. Indeed, we have the following equality $v_n \cdot\Sigma'=1+a_k+a_n v^2_n=0$, and like previously, the other coefficients are left unchanged.
    So we can also compute in this case
    \begin{align*}
        (\Sigma')^2=v'_0 \cdot\Sigma'&=\sum_{v_j \text{ incident to } v'_0 \text{ in } \Gvo'} a_j(v'_0 v_j)+(v'_0)^2+a_n (v'_0 v_n)\\
        &=\sum_{v_j \text{ incident to } v'_0 \text{ in } \Gvo'} a_j(v'_0 v_j)+(v^2_0 -1)+(a_k +1)\\
        &=\sum_{v_j \text{ incident to } v'_0 \text{ in } \Gvo'} a_j(v'_0 v_j)+v^2_0+a_k\\
        &=\sum_{v_j \text{ incident to } v_0 \text{ in } \Gvo} a_j(v_0 v_j)+v^2_0\\
        &=v_0 \cdot\Sigma=\Sigma^2
    \end{align*}
    since $v_k$ is incident to the distinguished vertex in $\Gvo$ but not in $\Gvo'$, and
    \begin{align*}
        L_0(\Sigma')&=2+(v'_0)^2+\sum^{n}_{j=1}a'_j(2+(v'_j)^2)\\
        &=2+(v'_0)^2+\sum^{n}_{\substack{j=1\\j\neq k,n}}a_j(2+v^2_j)+a'_k(2+(v'_k)^2)+a'_n(2+(v'_n)^2)\\
        &=1+v^2_0+\sum^{n}_{\substack{j=1\\j\neq k,n}}a_j(2+v^2_j)+a_k(1+v^2_k)+(a_k+1)\\
        &=2+v^2_0+\sum^{n-1}_{j=1}a_j(2+v^2_j)\\
        &=L_0(\Sigma).
\end{align*}
\end{proof}

\begin{remark}
    We can use the results of \cite{OSS14a} to see that if two plumbing graphs $\gammavo$ and $\Gamma'_{v'_0}$ represent the same pair $(Y,L)$ and are related by blow-ups of type I, IV and V exclusively, then the element $\leg$ is an invariant of the pair $(Y,L)$. First, following on the conventions in \cite{OSS14a}, denote by $\Gamma^d_{v_0}$ the graph we obtain from $\gammavo$ by adding a vertex with framing $-1$ without any edges, and denote by $\Gamma^+_{v_0}$ the same graph with an edge from the new vertex to $v_0$. In other words, the first graph corresponds to a blow-up of type I, while the second corresponds to a type IV. We have that the master complexes of $\Gamma^+_{v_0}$, $\Gamma^d_{v_0}$, and of the original graph $\gammavo$ are equal.
    
    As for the type V blow-up, we can make use of the filtered chain homotopies given in the Appendix of \cite{OSS14a} between the two filtered chain complexes.
\end{remark}

\nocite{*}
\bibliographystyle{plain}
\bibliography{references}
\bigskip
\noindent\textit{E-mail address}: \texttt{zampa.sarah@renyi.hu}

\bigskip
\textsc{Department of Geometry and Algebra, Institute of Mathematics, Budapest University of Technology and Economics, Műegyetem rkp. 3., H-1111 Budapest, Hungary}\par
\textsc{Department of Algebraic Geometry and Differential Topology, HUN-REN Alfréd R\'{e}nyi Institute of Mathematics, Reáltanoda u. 13-15., H-1053 Budapest, Hungary}

\end{document}